\documentclass[reqno, 12pt]{amsart}
\usepackage{a4wide}
\usepackage{amscd}
\usepackage{enumerate}
\input xy
\xyoption{all}

\newtheorem{thm}{Theorem}[section]

\newtheorem{lm}[thm]{Lemma}
\newtheorem{re}[thm]{Remark}

\newtheorem*{mT}{Main Theorem}

\numberwithin{equation}{section}

\begin{document}

\title[Real hypersurfaces with parallel structure Jacobi operator]{Real hypersurfaces in the complex quadric with parallel structure Jacobi operator}
\author{\textsc{Young Jin Suh}}

\address{Kyungpook National University \\ College of Natural Sciences\\  Department of Mathematics \\
Daegu 702-701\\ Republic of Korea}
\email{yjsuh@knu.ac.kr}
\date{}

\begin{abstract}
First we introduce the notion of parallel structure Jacobi operator for real hypersurfaces in the complex quadric $Q^m = SO_{m+2}/SO_mSO_2$ . Next we
give a complete classification of  real hypersurfaces in $Q^m = SO_{m+2}/SO_mSO_2$ with parallel structure Jacobi operator.
\end{abstract}

\maketitle
\thispagestyle{empty}

\footnote[0]{This work was supported by grant Proj. No. NRF-2015-R1A2A1A-01002459 from National Research Foundation of Korea. 
}

\section{Introduction}\label{section 1}

\vskip 6pt
In a class of Hermitian symmetric spaces of rank 2, usually we can give examples of
 Riemannian symmetric spaces $SU_{m+2}/S(U_2U_m)$ and $SU_{2,m}/S(U_2U_m)$, which are said to be complex two-plane Grassmannians and complex hyperbolic two-plane Grassmannians respectively (see \cite{BS1}, \cite{BSN}, \cite{SH} and \cite{S}). These are viewed as Hermitian symmetric spaces and quaternionic K\"{a}hler symmetric spaces equipped with the K\"{a}hler structure $J$ and  the quaternionic K\"{a}hler structure ${\mathfrak J}$ on $SU_{2,m}/S(U_2U_m)$. The rank of $SU_{2,m}/S(U_2U_m)$ is $2$ and there are exactly two types of singular tangent vectors $X$ of $SU_{2,m}/S(U_2U_m)$ which are characterized by the geometric properties $JX \in {\mathfrak J}X$ and $JX \perp {\mathfrak J}X$ respectively.
\par
\vskip 6pt

 As another kind of Hermitian symmetric space with rank $2$ of compact type different from the above ones, we can give the example of complex quadric $Q^m = SO_{m+2}/SO_mSO_2$, which is a complex hypersurface in complex projective space ${\Bbb C}P^m$ (see Berndt and Suh \cite{BS2}, and Smyth \cite{BS}). The complex quadric also can be regarded as a kind of real Grassmann manifold of compact type with rank 2 (see Kobayashi and Nomizu \cite{KO}).  Accordingly, the complex quadric admits both a complex conjugation structure $A$ and a K\"ahler structure $J$, which anti-commutes with each other, that is, $AJ=-JA$. Then for $m{\ge}2$ the triple $(Q^m,J,g)$ is a Hermitian symmetric space of compact type with rank 2 and its maximal sectional curvature is equal to $4$ (see Klein \cite{K} and Reckziegel \cite{R}).
 \par
 \vskip 6pt
In the complex projective space ${\mathbb C}P^m$, a full classification was obtained by Okumura in \cite{O}. He proved that the Reeb flow on a real hypersurface in ${\mathbb C}P^m = SU_{m+1}/S(U_mU_1)$ is isometric if and only if $M$ is an open part of a tube around a totally geodesic ${\mathbb C}P^k \subset {\mathbb C}P^m$ for some $k \in \{0,\ldots,m-1\}$. In the complex $2$-plane Grassmannian $G_2({\mathbb C}^{m+2})= SU_{m+2}/S(U_mU_2)$, Berndt and Suh \cite{BS1} proved the classification that the Reeb flow on a real hypersurface in $G_2({\mathbb C}^{m+2})$ is isometric if and only if $M$ is an open part of a tube around a totally geodesic $ G_2({\mathbb C}^{m+1})\subset G_2({\mathbb C}^{m+2})$.  Moreover, in \cite{S} we have shown that the Reeb flow on a real hypersurface in $SU_{2,m}/S(U_2U_m)$ is isometric if and only if $M$ is an open part of a tube around a totally geodesic $ SU_{2,m-1}/S(U_2U_{m-1})\subset SU_{2,m}/S(U_2U_m)$ .  In view of the previous two results a natural expectation might be that the classification involves at least the totally geodesic $Q^{m-1} \subset Q^m$. But, suprisingly, in the complex quadric $Q^m$ the situation is quite different from those ones. Recently, Berndt and Suh \cite{BS2} proved the following result:
 \par
 \vskip 6pt
\begin{thm}\label{thm 1}
Let $M$ be a real hypersurface in the complex quadric $Q^m$, $m\geq 3$. The Reeb flow on $M$ is isometric if and only if $m$ is even, say $m = 2k$, and $M$ is an open part of a tube around a totally geodesic ${\mathbb C}P^k \subset Q^{2k}$.
\end{thm}
\par
\vskip 6pt
On the other hand, Jacobi fields along geodesics of a given Riemannian manifold $(M,g)$ satisfy a well known differential equation. This equation naturally inspires the so-called Jacobi operator. That is, if $R$ denotes the curvature operator of $M$, and $X$ is tangent vector field to $M$, then the Jacobi operator $R_X{\in}End(T_xM)$ with respect to $X$ at $x{\in}M$, defined by $(R_XY)(x)=(R(Y,X)X)(x)$ for any $X{\in}T_xM$, becomes a self adjoint endomorphism of the tangent bundle $TM$ of $M$. Thus, each tangent vector field $X$ to $M$ provides a Jacobi operator $R_X$ with respect to $X$. In particular, for the Reeb vector field $\xi$, the Jacobi
operator $R_{\xi}$ is said to be a {\it structure Jacobi operator}.
\par
\vskip 6pt
Recently Ki, P\'erez, Santos and Suh \cite{KPSS} have investigated
      the Reeb parallel structure Jacobi operator in the complex space form $M_m(c)$, $c{\not =}0$ and have used it to study some principal curvatures for a tube over a totally geodesic submanifold. In particular, P\'erez, Jeong and Suh \cite{PJS} have investigated real hypersurfaces $M$ in $G_2({\mathbb C}^{m+2})$ with parallel structure Jacobi operator, that is, ${\nabla}_X{R}_{\xi}=0$ for any tangent vector field $X$ on $M$. Jeong, Suh and Woo \cite{JSW}  and  P\'erez and Santos \cite{PS} have generalized such a notion to the recurrent structure Jacobi operator, that is, $({\nabla}_X{R}_{\xi})Y={\beta}(X){R}_{\xi}Y$ for a certain $1$-form $\beta$ and any vector fields $X,Y$ on $M$ in $G_2({\mathbb C}^{m+2})$. Moreover, P\'erez, Santos and Suh \cite{PSS} have further investigated the property of the Lie $\xi$-parallel structure Jacobi operator in complex projective space ${\Bbb C}P^m$, that is, ${\mathcal L}_{\xi}R_{\xi}=0$.
\par
\vskip 6pt
When we consider a hypersurface $M$ in the complex quadric ${Q}^m$, the unit normal vector field $N$ of $M$ in $Q^m$ can be divided into two cases :  $N$ is $\frak A$-isotropic or $\frak A$-principal (See \cite{BS2}, \cite{BS3}, \cite{S1} and \cite{S2}). In the first case where $M$ has an $\frak A$-isotropic unit normal $N$, we have asserted in \cite{BS2} that $M$ is locally congruent to a tube over a totally geodesic ${\mathbb C}P^k$ in $Q^{2k}$.
In the second case when $N$ is $\frak A$-principal we have proved that $M$ is locally congruent to a tube over a totally geodesic and totally real submanifold $S^m$ in $Q^m$ (see \cite{BS3}).  In this paper we consider the case when the structure Jacobi operator $R_{\xi}$ of $M$ in $Q^m$ , that is, ${\nabla}_X{R}_{\xi}=0$ for any tangent vector field $X$ on $M$ is parallel, and we prove the following
\vskip 6pt
\begin{mT}\label{Main Theorem}\quad There do not exist any Hopf hypersurfaces in $Q^m$, $m{\ge}3$ with parallel stucture Jacobi operator.
\end{mT}
\par
\vskip 6pt

\medskip

\section{The complex quadric}\label{section 2}

For more background to this section we refer to \cite{BS2}, \cite{BS3}, \cite{K}, \cite{KO}, \cite{R}, \cite{S1} and \cite{S2} . The complex quadric $Q^m$ is the complex hypersurface in ${\mathbb C}P^{m+1}$ which is defined by the equation $z_1^2 + \cdots + z_{m+2}^2 = 0$, where $z_1,\ldots,z_{m+2}$ are homogeneous coordinates on ${\mathbb C}P^{m+1}$. We equip $Q^m$ with the Riemannian metric which is induced from the Fubini Study metric on ${\mathbb C}P^{m+1}$ with constant holomorphic sectional curvature $4$. The K\"{a}hler structure on ${\mathbb C}P^{m+1}$ induces canonically a K\"{a}hler structure $(J,g)$ on the complex quadric. For each $z \in Q^m$ we identify $T_z{\mathbb C}P^{m+1}$ with the orthogonal complement ${\mathbb C}^{m+2} \ominus {\mathbb C}z$ of ${\mathbb C}z$ in ${\mathbb C}^{m+2}$ (see Kobayashi and Nomizu \cite{KO}). The tangent space $T_zQ^m$ can then be identified canonically with the orthogonal complement ${\mathbb C}^{m+2} \ominus ({\mathbb C}z \oplus {\mathbb C}\rho)$ of ${\mathbb C}z \oplus {\mathbb C}\rho$ in ${\mathbb C}^{m+2}$, where $\rho \in \nu_zQ^m$ is a normal vector of $Q^m$ in ${\mathbb C}P^{m+1}$ at the point $z$.

The complex projective space ${\mathbb C}P^{m+1}$ is a Hermitian symmetric space of the special unitary group $SU_{m+2}$, namely ${\mathbb C}P^{m+1} = SU_{m+2}/S(U_{m+1}U_1)$. We denote by $o = [0,\ldots,0,1] \in {\mathbb C}P^{m+1}$ the fixed point of the action of the stabilizer $S(U_{m+1}U_1)$. The special orthogonal group $SO_{m+2} \subset SU_{m+2}$ acts on ${\mathbb C}P^{m+1}$ with cohomogeneity one. The orbit containing $o$ is a totally geodesic real projective space ${\mathbb R}P^{m+1} \subset {\mathbb C}P^{m+1}$. The second singular orbit of this action is the complex quadric $Q^m = SO_{m+2}/SO_mSO_2$. This homogeneous space model leads to the geometric interpretation of the complex quadric $Q^m$ as the Grassmann manifold $G_2^+({\mathbb R}^{m+2})$ of oriented $2$-planes in ${\mathbb R}^{m+2}$. It also gives a model of $Q^m$ as a Hermitian symmetric space of rank $2$. The complex quadric $Q^1$ is isometric to a sphere $S^2$ with constant curvature, and $Q^2$ is isometric to the Riemannian product of two $2$-spheres with constant curvature. For this reason we will assume $m \geq 3$ from now on.

For a unit normal vector $\rho$ of $Q^m$ at a point $z \in Q^m$ we denote by $A = A_\rho$ the shape operator of $Q^m$ in ${\mathbb C}P^{m+1}$ with respect to $\rho$. The shape operator is an involution on the tangent space $T_zQ^m$ and
$$T_zQ^m = V(A_\rho) \oplus JV(A_\rho),$$
where $V(A_\rho)$ is the $+1$-eigenspace and $JV(A_\rho)$ is the $(-1)$-eigenspace of $A_\rho$.  Geometrically this means that the shape operator $A_\rho$ defines a real structure on the complex vector space $T_zQ^m$, or equivalently, is a complex conjugation on $T_zQ^m$. Since the real codimension of $Q^m$ in ${\mathbb C}P^{m+1}$ is $2$, this induces an $S^1$-subbundle ${\mathfrak A}$ of the endomorphism bundle ${\rm End}(TQ^m)$ consisting of complex conjugations.

There is a geometric interpretation of these conjugations. The complex quadric $Q^m$ can be viewed as the complexification of the $m$-dimensional sphere $S^m$. Through each point $z \in Q^m$ there exists a one-parameter family of real forms of $Q^m$ which are isometric to the sphere $S^m$. These real forms are congruent to each other under action of the center $SO_2$ of the isotropy subgroup of $SO_{m+2}$ at $z$. The isometric reflection of $Q^m$ in such a real form $S^m$ is an isometry, and the differential at $z$ of such a reflection is a conjugation on $T_zQ^m$. In this way the family ${\mathfrak A}$ of conjugations on $T_zQ^m$ corresponds to the family of real forms $S^m$ of $Q^m$ containing $z$, and the subspaces $V(A) \subset T_zQ^m$ correspond to the tangent spaces $T_zS^m$ of the real forms $S^m$ of $Q^m$.

The Gauss equation for $Q^m \subset {\mathbb C}P^{m+1}$ implies that the Riemannian curvature tensor $\bar R$ of $Q^m$ can be described in terms of the complex structure $J$ and the complex conjugations $A \in {\mathfrak A}$:
\begin{eqnarray*}
{\bar R}(X,Y)Z & = & g(Y,Z)X - g(X,Z)Y + g(JY,Z)JX - g(JX,Z)JY - 2g(JX,Y)JZ \\
 & & + g(AY,Z)AX - g(AX,Z)AY + g(JAY,Z)JAX - g(JAX,Z)JAY.
\end{eqnarray*}

Then from the equation of Gauss the curvature tensor $R$ of $M$ in complex quadric $Q^m$ is defined so that
\begin{eqnarray*}
R(X,Y)Z & = & g(Y,Z)X - g(X,Z)Y + g({\phi}Y,Z){\phi}X - g({\phi}X,Z){\phi}Y - 2g({\phi}X,Y){\phi}Z \\
 & & + g(AY,Z)AX - g(AX,Z)AY + g(JAY,Z)JAX - g(JAX,Z)JAY\\
 & & + g(SY,Z)SX-g(SX,Z)SY,
\end{eqnarray*}
where $S$ denotes the shape operator of $M$ in $Q^m$.
\par
\vskip 6pt

Recall that a nonzero tangent vector $W \in T_zQ^m$ is called singular if it is tangent to more than one maximal flat in $Q^m$. There are two types of singular tangent vectors for the complex quadric $Q^m$:
\begin{itemize}
\item[1.] If there exists a conjugation $A \in {\mathfrak A}$ such that $W \in V(A)$, then $W$ is singular. Such a singular tangent vector is called ${\mathfrak A}$-principal.
\item[2.] If there exist a conjugation $A \in {\mathfrak A}$ and orthonormal vectors $X,Y \in V(A)$ such that $W/||W|| = (X+JY)/\sqrt{2}$, then $W$ is singular. Such a singular tangent vector is called ${\mathfrak A}$-isotropic.
\end{itemize}
For every unit tangent vector $W \in T_zQ^m$ there exist a conjugation $A \in {\mathfrak A}$ and orthonormal vectors $X,Y \in V(A)$ such that
\[
W = \cos(t)X + \sin(t)JY
\]
for some $t \in [0,\pi/4]$. The singular tangent vectors correspond to the values $t = 0$ and $t = \pi/4$. If $0 < t < \pi/4$ then the unique maximal flat containing $W$ is ${\mathbb R}X \oplus {\mathbb R}JY$.
Later we will need the eigenvalues and eigenspaces of the Jacobi operator $R_W = R(\cdot,W)W$ for a singular unit tangent vector $W$.
\begin{itemize}
\item[1.] If $W$ is an ${\mathfrak A}$-principal singular unit tangent vector with respect to $A \in {\mathfrak A}$, then the eigenvalues of $R_W$ are $0$ and $2$ and the corresponding eigenspaces are ${\mathbb R}W \oplus J(V(A) \ominus {\mathbb R}W)$ and $(V(A) \ominus {\mathbb R}W) \oplus {\mathbb R}JW$, respectively.
\item[2.] If $W$ is an ${\mathfrak A}$-isotropic singular unit tangent vector with respect to $A \in {\mathfrak A}$ and $X,Y \in V(A)$, then the eigenvalues of $R_W$ are $0$, $1$ and $4$ and the corresponding eigenspaces are ${\mathbb R}W \oplus {\mathbb C}(JX+Y)$, $T_zQ^m \ominus ({\mathbb C}X \oplus {\mathbb C}Y)$ and ${\mathbb R}JW$, respectively.
\end{itemize}

\section{Some general equations}\label{section 3}

Let $M$ be a  real hypersurface in $Q^m$ and denote by $(\phi,\xi,\eta,g)$ the induced almost contact metric structure. Note that $\xi = -JN$, where $N$ is a (local) unit normal vector field of $M$. The tangent bundle $TM$ of $M$ splits orthogonally into  $TM = {\mathcal C} \oplus {\mathbb R}\xi$, where ${\mathcal C} = {\rm ker}(\eta)$ is the maximal complex subbundle of $TM$. The structure tensor field $\phi$ restricted to ${\mathcal C}$ coincides with the complex structure $J$ restricted to ${\mathcal C}$, and $\phi \xi = 0$.
\par
\vskip 6pt
At each point $z \in M$ we define the maximal ${\mathfrak A}$-invariant subspace of $T_zM$, $z{\in}M$ as follows:

\[
{\mathcal Q}_z = \{X \in T_zM \mid AX \in T_zM\ {\rm for\ all}\ A \in {\mathfrak A}_z\}.
\]

\begin{lm}(see \cite{S1})\label{lemma 3.1}
For each $z \in M$ we have
\begin{itemize}
\item[(i)] If $N_z$ is ${\mathfrak A}$-principal, then ${\mathcal Q}_z = {\mathcal C}_z$.
\item[(ii)] IF $N_z$ is not ${\mathfrak A}$-principal, there exist a conjugation $A \in {\mathfrak A}$ and orthonormal vectors $X,Y \in V(A)$ such that $N_z = \cos(t)X + \sin(t)JY$ for some $t \in (0,\pi/4]$.
Then we have ${\mathcal Q}_z = {\mathcal C}_z \ominus {\mathbb C}(JX + Y)$.
\end{itemize}
\end{lm}
\par
\vskip 6pt
We now assume that $M$ is a Hopf hypersurface. Then we have
\[
S\xi = \alpha \xi
\]
with the smooth function $\alpha = g(S\xi,\xi)$ on $M$.
When we consider a transform $JX$ of the Kaehler structure $J$ on $Q^m$ for any vector field $X$ on $M$ in $Q^m$, we may put
$$JX={\phi}X+{\eta}(X)N$$
for a unit normal $N$ to $M$. We now consider the Codazzi equation
\begin{equation*}
\begin{split}
g((\nabla_XS)Y - (\nabla_YS)X,Z) & =  \eta(X)g(\phi Y,Z) - \eta(Y) g(\phi X,Z) - 2\eta(Z) g(\phi X,Y) \\
& \quad \ \   + g(X,AN)g(AY,Z) - g(Y,AN)g(AX,Z)\\
& \quad \ \   + g(X,A\xi)g(J AY,Z) - g(Y,A\xi)g(JAX,Z).
\end{split}
\end{equation*}
Putting $Z = \xi$ we get
\begin{equation*}
\begin{split}
g((\nabla_XS)Y - (\nabla_YS)X,\xi) & =   - 2 g(\phi X,Y) \\
& \quad \ \  + g(X,AN)g(Y,A\xi) - g(Y,AN)g(X,A\xi)\\
& \quad \ \  - g(X,A\xi)g(JY,A\xi) + g(Y,A\xi)g(JX,A\xi).
\end{split}
\end{equation*}
On the other hand, we have
\begin{eqnarray*}
 & & g((\nabla_XS)Y - (\nabla_YS)X,\xi) \\
& = & g((\nabla_XS)\xi,Y) - g((\nabla_YS)\xi,X) \\
& = & (X\alpha)\eta(Y) - (Y\alpha)\eta(X) + \alpha g((S\phi + \phi
S)X,Y) - 2g(S \phi SX,Y).
\end{eqnarray*}
Comparing the previous two equations and putting $X = \xi$ yields
$$
Y\alpha  =  (\xi \alpha)\eta(Y)  - 2g(\xi,AN)g(Y,A\xi) +
2g(Y,AN)g(\xi,A\xi).
$$
Reinserting this into the previous equation yields
\begin{eqnarray*}
 & & g((\nabla_XS)Y - (\nabla_YS)X,\xi) \\
& = &  - 2g(\xi,AN)g(X,A\xi)\eta(Y) + 2g(X,AN)g(\xi,A\xi)\eta(Y) \\
& &  + 2g(\xi,AN)g(Y,A\xi)\eta(X) - 2g(Y,AN)g(\xi,A\xi)\eta(X) \\
& & + \alpha g((\phi S + S\phi)X,Y) - 2g(S \phi SX,Y) .
\end{eqnarray*}
Altogether this implies
\begin{eqnarray*}
0 & = & 2g(S \phi SX,Y) - \alpha g((\phi S + S\phi)X,Y) - 2 g(\phi X,Y) \\
& & + g(X,AN)g(Y,A\xi) - g(Y,AN)g(X,A\xi)\\
& & - g(X,A\xi)g(JY,A\xi) + g(Y,A\xi)g(JX,A\xi)\\
& & + 2g(\xi,AN)g(X,A\xi)\eta(Y) - 2g(X,AN)g(\xi,A\xi)\eta(Y) \\
& &  - 2g(\xi,AN)g(Y,A\xi)\eta(X) + 2g(Y,AN)g(\xi,A\xi)\eta(X).
\end{eqnarray*}
At each point $z \in M$ we can choose $A \in {\mathfrak A}_z$ such that
\[ N = \cos(t)Z_1 + \sin(t)JZ_2 \]
for some orthonormal vectors $Z_1,Z_2 \in V(A)$ and $0 \leq t \leq \frac{\pi}{4}$ (see Proposition 3 in \cite{R}). Note that $t$ is a function on $M$.
First of all, since $\xi = -JN$, we have
\begin{eqnarray*}
N & = & \cos(t)Z_1 + \sin(t)JZ_2, \\
AN & = & \cos(t)Z_1 - \sin(t)JZ_2, \\
\xi & = & \sin(t)Z_2 - \cos(t)JZ_1, \\
A\xi & = & \sin(t)Z_2 + \cos(t)JZ_1.
\end{eqnarray*}
This implies $g(\xi,AN) = 0$ and hence
\begin{eqnarray*}
0 & = & 2g(S \phi SX,Y) - \alpha g((\phi S + S\phi)X,Y) - 2 g(\phi X,Y) \\
& & + g(X,AN)g(Y,A\xi) - g(Y,AN)g(X,A\xi)\\
& & - g(X,A\xi)g(JY,A\xi) + g(Y,A\xi)g(JX,A\xi)\\
& &  - 2g(X,AN)g(\xi,A\xi)\eta(Y) + 2g(Y,AN)g(\xi,A\xi)\eta(X).
\end{eqnarray*}
The curvature tensor $R(X,Y)Z$ for a real hypersurface $M$ in $Q^m$ is given by
\begin{eqnarray*}
R(X,Y)Z&=&g(Y,Z)X-g(X,Z)Y+g({\phi}Y,Z){\phi}X\\
& &-g({\phi}X,Z){\phi}Y-2g({\phi}X,Y){\phi}Z\\
& &+g(AY,Z)AX-g(AX,Z)AY+g(JAY,Z)JAX\\
& &-g(JAX,Z)JAY+g(SY,Z)SX-g(SX,Z)SY.
\end{eqnarray*}
From this, putting $Y=Z={\xi}$ and using $g(A{\xi},N)=0$, a structure Jacobi operator is defined by
\begin{eqnarray*}
R_{\xi}(X)&=&R(X,{\xi}){\xi}\\
&=& X-{\eta}(X){\xi}+g(A{\xi},{\xi})AX-g(AX,{\xi})A{\xi}\\
& & -g(JAX,{\xi})JA{\xi}+g(S{\xi},{\xi})SX-g(SX,{\xi})S{\xi}.
\end{eqnarray*}
\par
\vskip 6pt
\section{A key lemma}\label{section 4}
\par
\vskip 6pt
The curvature tensor $R(X,Y)Z$ for a Hopf real hypersurface $M$ in $Q^m$ induced from the curvature tensor
of $Q^m$ is given in section 3. Now the structure Jacobi operator $R_{\xi}$ from section 3 can be rewritten as follows:

\begin{equation}\label{e41}
\begin{split}
R_{\xi}(X)=&R(X,{\xi}){\xi}\\
=&X-{\eta}(X){\xi}+{\beta}AX-g(AX,{\xi})A{\xi}-g(AX,N)AN\\
&+{\alpha}SX-g(SX,{\xi})S{\xi},
\end{split}
\end{equation}
where we have put ${\alpha}=g(S{\xi},{\xi})$ and ${\beta}=g(A{\xi},{\xi})$, because we assume that $M$ is Hopf.  The Reeb vector field ${\xi}=-JN$ and the anti-commuting property $AJ=-JA$ gives that the function ${\beta}$ becomes ${\beta}=-g(AN,N)$. When this function ${\beta}=g(A{\xi},{\xi})$ identically vanishes, we say that a real hypersurface $M$ in $Q^m$ is $\frak A$-isotropic as in section 2.

Here we use the assumption of parallel structure Jacobi operator. Then (\ref{e41}) gives that
\begin{equation}\label{e42}
\begin{split}
0=&{\nabla}_{Y}R_{\xi}(X)={\nabla}_Y(R_{\xi}(X))-R_{\xi}({\nabla}_YX)\\
=&-({\nabla}_Y{\eta})(X){\xi}-{\eta}(X){\nabla}_Y{\xi}+(Y{\beta})AX\\
&+{\beta}\{{\bar{\nabla}}_Y(AX)-A{\nabla}_YX\}-g(X,{\bar{\nabla}}_Y(A{\xi}))A{\xi}\\
&-g(X,A{\xi}){\bar{\nabla}}_Y(A{\xi})-g(X,{\bar{\nabla}}_Y(AN))AN-g(X,AN){\bar{\nabla}}_Y(AN)\\
&+(Y{\alpha})SX+{\alpha}({\nabla}_YS)X-Y({\alpha}^2){\eta}(X){\xi}-{\alpha}^2({\nabla}_Y{\eta})(X){\xi}-{\alpha}^2{\eta}(X){\nabla}_Y{\xi}\\
=&-g({\phi}SY,X){\xi}-{\eta}(X){\phi}SY+(Y{\beta})AX\\
&+{\beta}\{q(Y)AX+g(SX,Y)AN\}-g(X,q(Y)A{\xi}+A{\phi}SY+{\alpha}{\eta}(Y)AN)A{\xi}\\
&-g(X,A{\xi})\{q(Y)A{\xi}+A{\phi}SY+{\alpha}{\eta}(Y)AN\}-g(X,q(Y)AN-ASY)AN\\
&-g(X,AN)\{q(Y)AN-ASY\}+(Y{\alpha})SX+{\alpha}({\nabla}_YS)X-Y({\alpha}^2){\eta}(X){\xi}\\
&-{\alpha}^2({\nabla}_Y{\eta})(X){\xi}-{\alpha}^2{\eta}(X){\nabla}_Y{\xi},
\end{split}
\end{equation}
where we have used the following formulae
\begin{equation*}
\begin{split}
{\bar{\nabla}}_Y(A{\xi})=&({\bar{\nabla}}_YA){\xi}+A({\bar{\nabla}}{\xi})\\
=&q(Y)A{\xi}+A{\phi}SY+g(SY,{\xi})AN,
\end{split}
\end{equation*}

\begin{equation*}
{\bar{\nabla}}_Y(AN)=({\bar\nabla}_YA)N+A{\bar{\nabla}}N=q(Y)AN-ASY,
\end{equation*}

and
\begin{equation*}
{\bar{\nabla}}_Y(AX)=({\bar\nabla}_YA)X+A{\bar{\nabla}}_YX=q(Y)AX+A({\nabla}_YX+{\sigma}(X,Y)).
\end{equation*}

From this, by taking inner product with the unit normal $N$, we have
\begin{equation*}
\begin{split}
0=&(Y{\beta})g(AX,N)+{\beta}\{q(Y)g(AX,N)+g(SX,Y)g(AN,N)\}\\
&-g(X,A{\xi})\{g(A{\phi}SY,N)+{\alpha}{\eta}(Y)g(AN,N)\}-g(X,q(Y)AN-ASY)g(AN,N)\\
&-q(Y)g(AN,N)g(X,AN)+g(X,AN)g(ASY,N)
\end{split}
\end{equation*}
Then by putting $X={\xi}$ and using $g(A{\xi},N)=0$, we have
\begin{equation}\label{e43}
{\beta}g(A{\phi}SY, N)+{\beta}g({\xi},ASY)=0.
\end{equation}

On the other hand, we know that
\begin{equation*}
\begin{split}
g({\xi},ASY)=&-g(JN,ASY)=g(N,JASY)=-g(N,AJSY)\\
=&-g(N,A{\phi}SY)-{\eta}(SY)g(N,AN).
\end{split}
\end{equation*}
Substituting this one into (\ref{e43}), we have
$${\beta}\{g(A{\phi}SY,N)+{\eta}(SY)g(N,AN)\}-{\beta}g(A{\phi}SY,N)=0.$$
Putting $Y={\xi}$ gives ${\alpha}{\beta}g(N,AN)=-{\alpha}{\beta}^2=0$, which implies $t=\frac{\pi}{4}$, if the Reeb function
$\alpha$ is non-vanishing. Then the unit normal vector field $N$ becomes
$$N=\frac{1}{\sqrt 2}(Z_1+JZ_2)$$
for $Z_1,Z_2{\in}V(A)$ as in section 3, that is, the unit normal $N$ is $\frak A$-isotropic .  Now we can prove the following
\par
\vskip 6pt
\begin{lm}
Let $M$ be a Hopf real hypersurface in complex quadric $Q^m$, $m{\ge}3$, with parallel structure Jacobi operator. Then the unit normal vector field $N$ is $\frak A$-principal or
$\frak A$-isotropic.
\end{lm}

\begin{proof}
\par
When the Reeb function $\alpha$ is non-vanishing, the unit normal $N$ is $\frak A$-isotropic as mentioned above.  When the Reeb function $\alpha$ identically vanishes, let us show that
$N$ is $\frak A$-isotropic or $\frak A$-principal.
In order to do this, from the condition of Hopf, we can differentiate $S{\xi}={\alpha}{\xi}$ and use the equation of Codazzi in section 3, then we get the formula
$$
Y\alpha  =  (\xi \alpha)\eta(Y)  - 2g(\xi,AN)g(Y,A\xi) +
2g(Y,AN)g(\xi,A\xi).
$$
From the assumption of ${\alpha}=0$ combined with the fact $g({\xi},AN)=0$ proved in section 3, we deduce $g(Y,AN)g({\xi},A{\xi})=0$ for any $Y{\in}T_xM$, $x{\in}M$.
This gives that the vector $AN$ is normal, that is, $AN=g(AN,N)N$ or $g(A{\xi},{\xi})=0$, which implies that the unit normal $N$ is $\frak A$-principal or $\frak A$-isotropic, respectively.
This completes the proof of our Lemma.
\end{proof}

By virtue of this Lemma, we distinguish between two classes of real hypersurfaces in complex quadric $Q^n$ with parallel structure Jacobi operator : those that have $\frak A$-principal unit normal, and those that have
$\frak A$-isotropic unit normal vector field $N$. We treat the respective cases in sections $5$ and $6$.

\section{Parallel structure Jacobi operator with $\frak A$-principal normal}\label{section 4}
In this section we consider a real hypersurface $M$ in a complex quadric with $\frak A$-principal unit normal vector field. Then the unit normal vector field $N$ satisfies $AN=N$
for a complex conjuagation $A{\in}{\frak A}$.

Then the structure Jacobi operator $R_{\xi}$ is given by
\begin{equation}\label{e51}
{R}_{\xi}(X)=X-2{\eta}(X){\xi}-AX+g(S{\xi},{\xi})SX-g(SX,{\xi})S{\xi}.
\end{equation}
Since we assume that $M$ is Hopf, (\ref{e51}) becomes
\begin{equation}\label{e52}
{R}_{\xi}(X)=X-2{\eta}(X){\xi}-AX+{\alpha}SX-{\alpha}^2{\eta}(X){\xi}.
\end{equation}
\par
\vskip 6pt
By the assumption that the structure Jacobi operator $R_{\xi}$ is parallel, the derivative of $R_{\xi}$ along any tangent vector field $Y$ on $M$ is given by
\begin{equation}\label{e53}
\begin{split}
0=&({\nabla}_YR_{\xi})(X)={\nabla}_Y(R_{\xi}(X))-R_{\xi}({\nabla}_YX)\\
=&-2\{({\nabla}_Y{\eta})(X){\xi}+{\eta}(X){\nabla}_Y{\xi}\}-({\nabla}_YA)X+(Y{\alpha})SX\\
&+{\alpha}({\nabla}_YS)X-(Y{\alpha}^2){\eta}(X){\xi}-{\alpha}^2({\nabla}_Y{\eta})(X){\xi}-{\alpha}^2{\eta}(X){\nabla}_Y{\xi}.
\end{split}
\end{equation}
\par

We can write
$$AY=BY+{\rho}(Y)N,$$
where $BY$ denotes the tangential component of $AY$ and ${\rho}(Y)=g(AY,N)=g(Y,AN)=g(Y,N)=0$. So we have $AY=BY$ for any vector field $Y$ on $M$ in $Q^m$.
Then it follows

\begin{equation}\label{e54}
\begin{split}
({\nabla}_YA)X=&({\bar{\nabla}}_YA)X+A\{{\nabla}_YX+{\sigma}(Y,X)\}\\
&-{\sigma}(Y,AX)-A{\nabla}_YX\\
=&q(Y)AX+A{\sigma}(Y,X)-{\sigma}(Y,AX)\\
=&q(Y)AX+g(SX,Y)AN-g(SY,AX)N,
\end{split}
\end{equation}
where we have used the Gauss and Weingarten formulae. From this, together with (\ref{e53}) and using that $\frak A$-principal, we have
\begin{equation}\label{e55}
\begin{split}
0=&({\nabla}_YR_{\xi})(X)\\
=&-(2+{\alpha}^2)\{({\nabla}_Y{\eta})(X){\xi}+{\eta}(X){\nabla}_Y{\xi}\}\\
&-\{q(Y)AX+g(SX,Y)N-g(SY,AX)N\}\\
&+(Y{\alpha})SX+{\alpha}({\nabla}_YS)X-(Y{\alpha}^2){\eta}(X){\xi}.
\end{split}
\end{equation}
From this, taking inner product of (\ref{e55}) with the unit normal vector field $N$, we have
\begin{equation*}
SAX=SX
\end{equation*}
for any vector field $X$ on $M$. By putting $X=-{\xi}$ and using that $M$ is Hopf,  we have
$$
SA{\xi}=S{\xi}={\alpha}{\xi}.$$
But the left side becomes $SA{\xi}=-S{\xi}=-{\alpha}{\xi}$. So it follows that the Reeb function $\alpha$ vanishes.
Then (\ref{e55}) can be rearranged as follows:
\begin{equation}\label{e56}
\begin{split}
0=&-2\{({\nabla}_Y{\eta})(X)+{\eta}(X){\nabla}_Y{\xi}\}-q(Y)AX\\
=&-2\{g({\phi}SY,X){\xi}+{\eta}(X){\phi}SY\}-q(Y)AX.
\end{split}
\end{equation}
From this, by putting $X={\xi}$, it gives
$$-2{\phi}SY=q(Y)A{\xi}=-q(Y){\xi}.$$
So $q(Y)=0$ implies ${\phi}SY=0$, which gives $SY={\alpha}{\eta}(Y){\xi}$, that is, $M$ is totally $\eta$-umbilical, in which case the shape operator commutes with
the structure tensor $\phi$. Then by a theorem due to Berndt and Suh \cite{BS2}, $M$ is locally congruent to a tube of radius $r$, $0<r<\frac{\pi}{2}$, over a totally geodesic
complex submanifold ${\Bbb C}P^k$ in $Q^{2k}$, $m=2k$.  Morever, in this case the uinit normal $N$ is $\frak A$-isotropic, not $\frak A$-principal.
This rules out the existence of an $\frak A$-principal unit normal $N$ together with a parallel structure Jacobi operator. So we give a proof of our main theorem with $\frak A$-principal unit normal $N$.

\vskip 6pt
\section{Parallel structure Jacobi operator with $\frak A$-isotropic normal}\label{section 6}
\par
\vskip 6pt
In this section we assume that the unit normal vector field $N$ is $\frak A$-isotropic.  Then the normal vector field $N$ can be written as
$$N=\frac{1}{\sqrt 2}(Z_1+JZ_2)$$
for $Z_1,Z_2{\in}V(A)$, where $V(A)$ denotes a $+1$-eigenspace of the complex conjugation $A{\in}{\frak A}$. Then it follows that
$$AN=\frac{1}{\sqrt 2}(Z_1-JZ_2),\ AJN=-\frac{1}{\sqrt 2}(JZ_1+Z_2),\  \text{and}\  JN=\frac{1}{\sqrt 2}(JZ_1-Z_2).$$
Then it gives that
$$g({\xi},A{\xi})=g(JN,AJN)=0, g({\xi},AN)=0\ \text{and}\ g(AN,N)=0.$$
By virtue of these formulas for $\frak A$-isotropic unit normal, the structure Jacobi operator can be defined so that
\begin{equation}\label{e61}
\begin{split}
{R}_{\xi}(X)=&R(X,{\xi}){\xi}\\
=&X-{\eta}(X){\xi}-g(AX,{\xi})A{\xi}-g(JAX,{\xi})JA{\xi}\\
&+g(S{\xi},{\xi})SX-g(SX,{\xi})S{\xi}.
\end{split}
\end{equation}
\par
\vskip 6pt
On the other hand, we know that $JA{\xi}=-JAJN=AJ^2N=-JN$, and $g(JAX,{\xi})=-g(AX,J{\xi})=-g(AX,N)$. Now the structure Jacobi operator $R_{\xi}$ can be rearranged as follows:
\begin{equation}\label{e62}
\begin{split}
{R}_{\xi}(X)=&X-{\eta}(X){\xi}-g(AX,{\xi})A{\xi}-g(X,AN)AN\\
&+{\alpha}SX-{\alpha}^2{\eta}(X){\xi}.
\end{split}
\end{equation}
Differentiating (\ref{e62}) we obtain
\begin{equation}\label{e63}
\begin{split}
{\nabla}_YR_{\xi}(X)=&{\nabla}_Y(R_{\xi}(X))-R_{\xi}({\nabla}_YX)\\
=&-({\nabla}_Y{\eta})(X){\xi}-{\eta}(X){\nabla}_Y{\xi}-g(X,{\nabla}_Y(A{\xi}))A{\xi}\\
&-g(X,A{\xi}){\nabla}_Y(A{\xi})-g(X,{\nabla}_Y(AN))AN-g(X,AN){\nabla}_Y(AN)\\
&+(Y{\alpha})SX+{\alpha}({\nabla}_YS)X-(Y{\alpha}^2){\eta}(X){\xi}\\
&-{\alpha}^2({\nabla}_Y{\eta})(X){\xi}-{\alpha}^2{\eta}(X){\nabla}_Y{\xi}.
\end{split}
\end{equation}
Here let us use the equation of Gauss and Weingarten formula as follows:
\begin{equation*}
\begin{split}
{\nabla}_Y(AN)=&{\bar{\nabla}}_Y(AN)-{\sigma}(Y,AN)\\
=&({\bar\nabla}_YA)N+A{\bar\nabla}_YN-{\sigma}(Y,AN)\\
=&q(Y)JAN-ASY-{\sigma}(Y,AN),\\
=&q(Y)A{\xi}-ASY-{\sigma}(Y,AN),
\end{split}
\end{equation*}
and
\begin{equation*}
\begin{split}
{\nabla}_Y(A{\xi})=&{\bar{\nabla}}_Y(A{\xi})-{\sigma}(Y,A{\xi})\\
=&({\bar\nabla}_YA){\xi}+A{\bar\nabla}_Y{\xi}-{\sigma}(Y,A{\xi})\\
=&q(Y)JA{\xi}+A\{{\phi}SY+{\eta}(SY)N\}-{\sigma}(Y,A{\xi})\\
=&-q(Y)AN+A{\phi}SY+{\eta}(SY)AN-{\sigma}(Y,A{\xi})
\end{split}
\end{equation*}
Substituting these formulas into (\ref{e63}) and using the assumption of parallel structure Jacobi operator, we have
\begin{equation}\label{e64}
\begin{split}
0=&{\nabla}_YR_{\xi}(X)\\
=&-g({\phi}SY,X){\xi}-{\eta}(X){\phi}SY\\
&-\{q(Y)g(A{\xi},X)+g(A{\phi}SY,X)+g(SY,{\xi})g(AN,X)\}A{\xi}\\
&-g(X,A{\xi})\{q(Y)A{\xi}+A{\phi}SY+g(SY,{\xi})AN\\
&-g(SY,A{\xi})N\}-\{q(Y)g(X,AN)-g(X,ASY)\}AN\\
&-g(X,AN)\{q(Y)AN-ASY-g(SY,AN)N\}\\
&+(Y{\alpha})SX+{\alpha}({\nabla}_YS)X-(Y{\alpha}^2){\eta}(X){\xi}\\
&-{\alpha}^2g({\phi}SY,X){\xi}-{\alpha}^2{\eta}(X){\phi}SY.
\end{split}
\end{equation}
From this, taking inner product with the Reeb vector field $\xi$, we have
\begin{equation}\label{e65}
\begin{split}
0=&-g({\phi}SY,X)-g(X,A{\xi})g(A{\phi}SY,{\xi})+g(X,AN)g(ASY,{\xi})\\
&+(Y{\alpha}){\alpha}{\eta}(X)+{\alpha}g(({\nabla}_YS)X,{\xi})\\
&-(Y{\alpha}^2){\eta}(X)-{\alpha}^2g({\phi}SY,X).
\end{split}
\end{equation}
Here by the assumption that $M$ is Hopf we can use the following
\begin{equation*}
({\nabla}_YS){\xi}={\nabla}_Y(S{\xi})-S({\nabla}_Y{\xi})=(Y{\alpha}){\xi}+{\alpha}{\phi}SY-S{\phi}SY.
\end{equation*}
Then it follows that
\begin{equation}\label{e66}
{\alpha}g(({\nabla}_YS)X,{\xi})=g({\alpha}(Y{\alpha}){\xi}+{\alpha}^2{\phi}SY-{\alpha}S{\phi}SY,X).
\end{equation}
Taking inner product of (\ref{e64}) with the unit normal $N$, it follows that
\begin{equation}\label{e67}
\begin{split}
0=&-g(X,A{\xi})g(A{\phi}SY,N)+g(X,A{\xi})g(SY,A{\xi})\\
&+g(X,AN)g(ASY,N)+g(X,AN)g(SY,AN).
\end{split}
\end{equation}
From this, putting $X=AN$ and using that $N$ is $\frak A$-isotropic, we have $SAN=0$. This also gives $S{\phi}A{\xi}=0$.
\par
\vskip 6pt
On the other hand, one of the terms $g(SY,A{\xi})$ in (\ref{e64}) becomes
\begin{equation*}
\begin{split}
g(SY,A{\xi})=-g(SY, AJN)=g(SY, JAN)=g(SY, {\phi}AN+{\eta}(AN)N)=-g(A{\phi}SY,N).
\end{split}
\end{equation*}
Substituting this term into (\ref{e67}) gives $S{\phi}AN=0$. Summing up these formulas, we can write
\begin{equation}\label{e68}
SA{\xi}=0,\ SAN=0,\ S{\phi}A{\xi}=0, \text{and}\ S{\phi}AN=0.
\end{equation}
Putting $X={\xi}$ in (\ref{e64}), and using (\ref{e66}) and (\ref{e68}), we have
\begin{equation}\label{e69}
{\phi}SY=-{\alpha}S{\phi}SY.
\end{equation}
In the case when $N$ is $\frak A$-isotropic, Berndt and Suh \cite{BS2} have proved that the Reeb function $\alpha$ is constant. So, we divide into the two cases that ${\alpha}=0$ and ${\alpha}{\not =}0$. For the first case, (\ref{e69}) gives ${\phi}SY=0$, which implies $SY={\alpha}{\eta}(Y){\xi}$, that is, $M$ is totally $\eta$-umbilical.
So the structure tensor commutes with the shape operator, $S{\phi}={\phi}S$. Then by Berndt and Suh \cite{BS2} $M$ is locally congruent to a tube of radius $r$ over a totally geodesic ${\Bbb C}P^k$ in $Q^{2k}$.
But the Reeb function $\alpha$ of this tube never vanishing. So such a case ${\alpha}=0$ can not happen. Finally we consider the case ${\alpha}{\not =}0$.
\par
\vskip 6pt
On the other hand, on the distribution $\mathcal Q$ let us introduce an important formula mentioned in section 3 as follows:
\begin{equation}\label{e610}
2S{\phi}SY-{\alpha}({\phi}S+S{\phi})Y=2{\phi}Y
\end{equation}
for any tangent vector field $Y$ on $M$ in $Q^m$ (see also \cite{BS2} , pages 1350050-11).
So if $SY={\lambda}Y$ in (\ref{e610}), then $(2{\lambda}-{\alpha})S{\phi}Y=({\alpha}{\lambda}+2){\phi}Y$, which gives
\begin{equation}\label{e611}
S{\phi}Y=\frac{{\alpha}{\lambda}+2}{2{\lambda}-{\alpha}}{\phi}Y,
\end{equation}
because if $2{\lambda}-{\alpha}=0$, then ${\alpha}{\lambda} + 2=0$. This implies ${\alpha}^2+4=0$, which gives us a contradiction.
By (\ref{e69}) and (\ref{e610}), we know that
$$-\frac{2+{\alpha}^2}{\alpha}{\phi}S-{\alpha}S{\phi}=2{\phi}.$$
From this, putting $SY={\lambda}Y$ and using (\ref{e611}), we know that
\begin{equation}\label{e612}
\begin{split}
S{\phi}Y=&-\frac{2{\lambda}+{\alpha}^2{\lambda}+2{\alpha}}{{\alpha}^2}{\phi}Y\\
=&\frac{{\alpha}{\lambda}+2}{2{\lambda}-{\alpha}}{\phi}Y.
\end{split}
\end{equation}
Then by a straightforward calculation, we get the following equation
$${\lambda}\{2({\alpha}^2+2){\lambda}+2{\alpha}\}=0.$$
This means ${\lambda}=0$ or ${\lambda}=-\frac{\alpha}{{\alpha}^2+2}$. When ${\lambda}=0$, by (\ref{e612}), $S{\phi}Y=-\frac{2}{\alpha}{\phi}Y$. Then $\frac{2}{\alpha}=\frac{\alpha}{{\alpha}^2+2}$, which gives ${\alpha}^2+4=0$.
This is again a contradiction. So we can assume that the other principal curvature is $-\frac{\alpha}{{\alpha}^2+2}$. Now let us denote this princpal curvature by the function $\beta$.
Accordingly, the shape operator $S$ can be expressed as

\begin{equation*}
S=\begin{bmatrix}
{\alpha} & 0 & 0 & 0 & {\cdots} & 0 & 0 & {\cdots} & 0\\
0 & 0 & 0 & 0 & {\cdots} & 0 & 0 & {\cdots} & 0 \\
0 & 0 & 0 & 0 & {\cdots} & 0 & 0 & {\cdots} & 0 \\
0 & 0 & 0 & {\beta} & {\cdots} & 0 & 0 & {\cdots} & 0\\
{\vdots} & {\vdots} & {\vdots} & {\vdots} & {\ddots} & {\vdots} & {\vdots} & {\cdots} & {\vdots}\\
0 & 0 & 0 & 0 & {\cdots} & {\beta} & 0 & {\cdots} & 0\\
0 & 0 & 0 & 0 & {\cdots} & 0 & {\beta} & {\cdots} & 0 \\
{\vdots} & {\vdots} & {\vdots} & {\vdots} & {\vdots} & {\vdots} & {\vdots} & {\ddots} & {\vdots}\\
0 & 0 & 0 & 0 & {\cdots} & 0 & 0 & {\cdots} & {\beta}
\end{bmatrix}
\end{equation*}
Let us consider the principal curvature $\beta$ such that $SY={\beta}Y$ for the formula (\ref{e69}). It follows that ${\beta}{\phi}Y=-{\alpha}{\beta}S{\phi}Y$. From this, together with the
expression for $S$, we have
\begin{equation*}
\begin{split}
 S{\phi}Y=&{\beta}{\phi}Y\\
 =&-\frac{\beta}{{\alpha}{\beta}}{\phi}Y=-\frac{1}{\alpha}{\phi}Y.
\end{split}
\end{equation*}
Then $-1={\alpha}{\beta}=-\frac{{\alpha}^2}{{\alpha}^2+2}$, which gives us a contradiction.
Accordingly, we conclude that any real hypersurfaces $M$ in $Q^m$ with $\frak A$-isotropic do not admit a parallel structure Jacobi operator.
\par
\vskip 6pt
\begin{re}
Let us check that a tube of radius $r$ $(0<r<\frac{\pi}{2})$ over a totally geodesic ${\Bbb C}P^k$ in $Q^{2k}$. Then by Berndt and Suh \cite{BS1} the tube has a commuting shape operator, that is, $S{\phi}={\phi}S$ and the unit normal $N$ is $\frak A$-isotropic and the Reeb curvature ${\alpha}=g(S{\xi},{\xi})$ is constant. By putting $Y={\xi}$ in (\ref{e64}) and using the properties $g(A{\xi},{\xi})=0$ and $g(AN,N)=0$, we have
\begin{equation}\label{e613}
\begin{split}
0=&-\{q({\xi})g(A{\xi},X)+{\alpha}g(AN,X)\}A{\xi}-g(X,A{\xi})\{q({\xi})A{\xi}+{\alpha}AN\}\\
&+\{q({\xi})g(X,AN)-{\alpha}g(X,A{\xi})\}AN+g(X,AN)\{q({\xi})AN-{\alpha}A{\xi}\}\\
&+{\alpha}({\nabla}_{\xi}S)X .
\end{split}
\end{equation}
Since on such a tube $SA{\xi}=0$ and $SAN=0$, by differentiating we know that $({\nabla}_{\xi}S)A{\xi}=0$. In fact, by the Gauss and Weingarten formulas and using $g(S{\xi},A{\xi})=0$ and ${\nabla}_{\xi}{\xi}=0$, we have
$${\nabla}_{\xi}(A{\xi})=q({\xi})A{\xi}+{\alpha}AN.$$
So we have $S{\nabla}_{\xi}(A{\xi})=q({\xi})SA{\xi}+{\alpha}SAN=0$. By using such a fact, taking inner product of (\ref{e613}) with $A{\xi}$, we have
$$0=q({\xi})g(A{\xi},X)+{\alpha}g(AN,X).$$
From this, by putting $X=AN{\in}T_xM$, $x{\in}M$, it follows that $2{\alpha}g(AN,AN)=0$, which gives ${\alpha}=0$. But this contradicts the poperties of the tube. So the tube does not admit a parallel structure Jacobi operator.
\end{re}

\par
\vskip 6pt
\begin{re}
When we consider that $M$ is locally congruent to a tube of radius $r$, $0<r<\frac{\pi}{2{\sqrt 2}}$, over a totally geodesic and totally real space form $S^m$ in $Q^m$. Then in Berndt and Suh \cite{BS3} it
 was shown that $M$ has three distinct constant principal curvatures ${\alpha}=-{\sqrt 2}\cot({\sqrt 2}r)$, ${\lambda}=0$ and ${\mu}={\sqrt 2}\tan ({\sqrt 2}r)$ with multiplicities $1$, $m-1$ and $m-1$ respectively.
 This is equivalent to ${\phi}S+S{\phi}=k{\phi}$, where $k$ is a constant $k{\not =}0$.  Moreover, the unit normal $N$ of $M$ in $Q^m$ is $\frak A$-principal, that is, $AN=N$, and $A{\xi}=-{\xi}$. If we assume that the structure Jacobi operator on $M$ is parallel, then as remarked in section 5 we know the following
 $$0=({\nabla}_YR_{\xi})X=-2\{g({\phi}SY,X){\xi}+{\eta}(X){\phi}SY\}-q(Y)AX$$
 for any vector fields $X$ and $Y$ on $M$. From this, by putting $X={\xi}$ gives $2{\phi}SY=-q(Y)A{\xi}=q(Y){\xi}$. This gives $q(Y)=0$ for any $Y$ on $M$, so it follows that ${\phi}SY=0$. This means that $SY={\alpha}{\eta}(Y){\xi}$, that is, the tube is totally $\eta$-umbilical. But the tube mentioned above is never $\eta$-umbilical. Accordingly, the tube does not admit any parallel structure Jacobi operator.
 \end{re}
\par
\vskip 6pt
\begin{re}
In \cite{S3} we have classified real hypersurfacees $M$ in complex quadric $Q^m$ with parallel Ricci tensor, according to whether the unit normal $N$ is $\mathfrak A$-principal or $\mathfrak A$-isotropic. When $N$ is $\mathfrak A$-principal, we proved a non-existence property for Hopf hypersurfaces in $Q^m$. For a Hopf real hypersurface $M$ in $Q^m$ with $\mathfrak A$-isotropic we have given a complete classification that it has {\it three distinct constant} principal curvatures.
\end{re}

\par

\end{document}